\documentclass{lmcs} 
\pdfoutput=1

\usepackage{lastpage}
\lmcsdoi{16}{4}{18}
\lmcsheading{}{\pageref{LastPage}}{}{}%
{Jan.~22,~2020}{Dec.~18,~2020}{}

\usepackage[utf8]{inputenc}


\usepackage{lineno,hyperref}
\usepackage{amssymb}
\usepackage{amsmath}
\usepackage{amsthm}

\usepackage[matrix,arrow]{xy}
\theoremstyle{plain} 

\newcommand{\ua}{\mathord{\uparrow}}
\newcommand{\da}{\mathord{\downarrow}}

\newcommand{\cl}{{\rm cl}}

\begin{document}
	
	\title[Open well-filtered spaces]{On open well-filtered spaces}
	
	\author[C.~Shen]{Chong Shen}	
	\address{ School of Mathematical Sciences,	Nanjing Normal University, Jiangsu, Nanjing,
		China}	
	\email{shenchong0520@163.com}  
	\thanks{The first author is sponsored by NSFC (11871097)  and the China Scolarship Council (File No. 201806030073)}	
	
	\author[X.~Xi]{Xiaoyong Xi}	
	\address{School of Mathematics and Statistics, Jiangsu Normal University, Jiangsu, Xuzhou, China}	
	\email{littlebrook@jsnu.edu.cn}  
	\thanks{The second author is sponsored by NSFC (1207188).}	
	
	\author[X.~Xu]{Xiaoquan Xu}	
	\address{School of Mathematics and Statistics, Minnan Normal University, Fujian, Zhangzhou, China}	
	\email{xiqxu2002@163.com}
	\thanks{The third author is sponsored by NSFC (11661057, 1207199), the Ganpo 555 project for leading talent of Jiangxi Provence and the Natural Science Foundation of Jiangxi Province, China (20192ACBL20045).}	
	
	\author[D.~Zhao]{Dongsheng Zhao}	
	\address{Mathematics and Mathematics Education, National Institute of Education,
		Nanyang Technological University,  1 Nanyang Walk, Singapore}	
	\email{dongsheng.zhao@nie.edu.sg}
	
	
	
	
	\begin{abstract}
		\noindent We introduce and study a new class of $T_0$ spaces, called open well-filtered spaces. The main results we prove include (i) every well-filtered space is an open well-filtered space; (ii) every core-compact open well-filtered  space is sober. As an immediate corollary, it follows that every core-compact well-filtered space is sober. This provides a different and relatively more straightforward method to answer the problem posed by Jia and Jung: is every core-compact well-filtered space sober?
	\end{abstract}
	
	\maketitle
	
	\section{Introduction}\label{S:one}
	
	The sobriety is one of the most important topological properties, particularly meaningful for $T_0$ spaces. It has been  used in the characterization of spectral spaces of commutative rings and the spaces  which are determined by their open set lattices.
	In domain theory, it was proved that the Scott space of every domain is sober  quite early on. Since then the investigation of the sobriety of Scott spaces of general directed complete posets led to many deep results. Heckmann introduced the well-filtered spaces and asked whether every well-filtered Scott space of a directed complete poset is sober \cite{Heckmann1990,Heckmann1992}.
	This question inspired intensive studies on the relationship between sobriety and  well-filteredness (see \cite{HGJX, Jia-Jung-2016,kou,zhao-xi-chen,xi-zhao-MSCS-well-filtered,Xi-Lawson-2017,wu-xi-xu-zhao-19}). A recent problem on this topic is whether every core-compact well-filtered space is sober,  posed by Jia and Jung \cite{Jiathesis}. The problem has been answered positively by Lawson, Wu and Xi \cite{LawsonXi2019} and Xu, Shen, Xi and Zhao \cite{Xu-Shen-Xi-Zhao2020}.

	In the current paper we first introduce a new class of topological spaces, called open well-filtered spaces,  which includes all well-filtered spaces. The open well-filtered spaces themselves may  deserve further study that will enrich the theory of $T_0$ topological spaces. We prove that
	(i) every well-filtered space is an open well-filtered space, and  (ii) every core-compact open well-filtered  space is sober. As an immediate implication, we obtain that every core-compact well-filtered space is sober, thus giving a relatively more straightforward
	method to answer Jia and Jung's problem \cite{Jiathesis}.

\section{Preliminaries}\label{S:two}
	
	This section is devoted to a brief review of some basic concepts and notations that will be used in the paper. For more details, see \cite{Engelking,redbook,Jean-2013}.
	
	Let $P$ be a poset.
	A nonempty subset $D$ of $P$ is \emph{directed} if every two
	elements in $D$ have an upper bound in $D$.  $P$ is called a
	\emph{directed complete poset}, or \emph{dcpo} for short, if for any
	directed subset $D\subseteq P$, $\bigvee D$ exists.

	Let $X$ be a $T_0$ space. A subset $A$ of $X$ is called \emph{saturated} if $A$ equals the intersection of all open sets containing it. The specialization order $\leq $ on $X$  is defined by $x\leq y$ if{}f $x\in \cl(\{y\})$, where $\cl$ is the closure operator.
	It is easy to show that a subset $A$ of  $X$ is saturated if and only if $A=\ua A=\{x\in X: x\geq a\text{ for some }a\in A\}$ with respect to the specialization order.

	A nonempty subset $A$ of $X$ is \emph{irreducible} if for any closed sets $F_1, F_2$ of $X$, $A\subseteq F_1\cup F_2$ implies $A\subseteq F_1$ or $A\subseteq F_2$.
	A $T_0$ space $X$ is called \emph{sober} if for any irreducible closed set $F$, $F=\da x=\cl(\{x\})$ for some $x\in X$.

	For a $T_0$ space $X$, we shall consider the following subfamilies of the power set $2^X$:
	
	$\begin{array}{ll}
	\mathcal Q(X), & \text{the set of all compact saturated subsets of }X;\\
	\mathcal S(X),  & \text{the set of all saturated subsets of } X;\\
	\mathcal O(X), & \text{the set of all open subsets of }X.
	\end{array}$

	For $A,B\subseteq X$, we say that \emph{$A$  is relatively compact in $B$}, denoted by $A\ll B$,  if $A\subseteq B$ and every open cover of $B$ contains a finite subcover of $A$.

	  We write
	$$\mathcal A\subseteq_{flt} 2^X\ \ (\mathcal Q(X), \mathcal S(X), \mathcal O(X), \text{ resp.})$$
	for that $\mathcal A$ is a $\ll$-filtered subfamily of $2^X$  ($\mathcal Q(X)$, $\mathcal S(X)$,  $\mathcal O(X)$,  resp.), that is, $\forall A_1,A_2\in\mathcal A$, there exists $A_3\in\mathcal A$ such that $A_3\ll  A_1,  A_2$.
	
		\begin{rem}(1) In general, for any $A,B\subseteq X$, that each open cover of $B$ contains a finite subcover of $A$ does not  imply $A\subseteq B$. For example, on the set $X=\{x,y\}$, consider the topology $\mathcal O(X)=\{\emptyset, \{x\},X\}$. Then $X$ is a $T_0$ space and $\{x\}\ll\{y\}$.
		Thus the requirement $A\subseteq B$ in the definition of $A\ll B$ is  not redundant.
		
		(2) For any $A\subseteq X$ and $B\in\mathcal O(X)$,  $A\ll B$ if and only if each open cover of $B$ contains a finite subcover of $A$. This is because
		$\{B\}$ is an open cover of $A$.
	
	(3) For any   $A\subseteq X$ and $Q\in\mathcal Q(X)$,   $A\ll Q$ if and only if $A\subseteq Q$. Hence, $\mathcal A\subseteq_{flt}\mathcal Q(X)$ if and only if $(\mathcal A,\supseteq)$ is a directed family.
	\end{rem}

	A $T_0$  space $X$ is called \emph{well-filtered}  if for any
	$\mathcal K\subseteq_{flt}\mathcal Q(X)$ and $U\in\mathcal O(X)$, $\bigcap\mathcal{K}\subseteq U$ implies $K\subseteq U$ for some $K\in\mathcal{K}$.
	We note that every sober space is well-filtered \cite{redbook}.
	
	In what follows, the symbol $\omega$ will denote  the smallest infinite ordinal, and for any set $X$,  the family of all finite subsets of $X$ is denoted by $X^{(<\omega)}$.
	
	\begin{defi}
		A $T_0$ space $X$ is called \emph{$\omega$-well-filtered}, if for any $\{K_n:n<\omega\}\subseteq_{flt} \mathcal Q(X)$ and  $U\in\mathcal O(X)$,
		$$\bigcap_{n<\omega}K_n\subseteq U \ \Rightarrow \  \exists n_0<\omega, K_{n_0}\subseteq U.$$
	\end{defi}
	
	\begin{prop}\label{chain}
		A $T_0$ space $X$ is $\omega$-well-filtered if and only if for any  descending chain $\{K_n:n<\omega\}\subseteq\mathcal Q(X)$, that is,
		$$K_0\supseteq K_1\supseteq K_2\supseteq\ldots\supseteq K_n\supseteq K_{n+1}\supseteq\ldots,$$
		and $U\in\mathcal O(X)$,
		$$\bigcap_{n<\omega}K_n\subseteq U\ \Rightarrow\ \exists n_0<\omega, \ K_{n_0}\subseteq U.$$
	\end{prop}
	
	\begin{proof}
		We only need to prove the sufficiency. Let $\mathcal K\subseteq_{flt}\mathcal Q(X)$ be a countable  family and $U\in\mathcal O(X)$ such that $\bigcap\mathcal K\subseteq U$.
		
		If the cardinality $|\mathcal K|<\omega$, i.e., $\mathcal K$ is a finite family,  then $\mathcal K$ contains a smallest element $Q$, and hence $Q=\bigcap\mathcal K\subseteq U$, completing the proof.

		Now assume $|\mathcal K|=\omega$. We may let $\mathcal K=\{K_n:n<\omega\}$. We  use induction on $n<\omega$ to  define a descending chain $\widehat{\mathcal K}=\left\{\widehat{K}_n:n<\omega\right\}$.
		Specifically, let $\widehat{K}_0=K_0$ and let $\widehat{K}_{n+1}\in\mathcal K$ be a lower bound of $\left\{K_{n+1},\widehat{K}_0, \widehat{K}_1,\widehat{K}_2\ldots,\widehat{K}_n\right\}$  under the inclusion order. Then $\widehat{\mathcal K}\subseteq \mathcal K$ is a descending chain and  $\widehat{K}_n\subseteq K_n$ for all $n<\omega$,  implying that  $\bigcap\widehat{\mathcal K}=\bigcap\mathcal K\subseteq U$.
		Then by assumption, there exists $n_0<\omega$ such that $\widehat{K}_{n_0}\subseteq U$, completing the proof.
	\end{proof}

	\begin{lem}\label{llrudin}
		Let $X$ be a $T_0$ space and $\mathcal A\subseteq_{flt} 2^X$.  Each  closed set $C\subseteq X$ that intersects all members of $\mathcal A$ contains a minimal (irreducible) closed subset $F$ that still intersects all members of  $\mathcal A$.
	\end{lem}
	\begin{proof}
		Let $\mathcal B:=\{B\in\mathcal C(X): \forall A\in\mathcal A, B\cap A\neq\emptyset\}$, where $\mathcal C(X)$ is the set of all closed subsets of $X$.
		
		(i) $\mathcal B\neq\emptyset$ because $C\in \mathcal B$.
		
		(ii) Let $\mathcal H\subseteq \mathcal B$ be a chain. We claim that $\bigcap\mathcal H\in\mathcal B$. Otherwise, there exists $A_0\in\mathcal A$ such that $A_0\cap\bigcap\mathcal H=\emptyset$. As $\mathcal A$ is $\ll$-filtered, there exists $A_1\in\mathcal A$ such that $A_1\ll A_0$. Since $\{X\setminus B: B\in\mathcal H\}$ is a directed open cover of $A_0$, there exists $B_0\in\mathcal H$ such that $A_1\subseteq X\setminus B_0$ implying $A_1\cap B_0=\emptyset$. This means that $B_0\notin\mathcal B$, contradicting  $B_0\in\mathcal H\subseteq\mathcal B$.
		
		By Zorn's Lemma, it turns out that there exists a minimal closed subset $F\subseteq C$  such that $F\cap A\neq\emptyset$ for all $A\in\mathcal A$.
		
		Now we show that $F$ is irreducible. If $F$ is not irreducible, then there exist closed sets $F_1, F_2$ such that $F=F_1\cup F_2$ but $F\neq F_1$ and $F\neq F_2$.
		Since $F_1, F_2$ are proper subsets of $F$,   by the minimality of  $F$, there exist $A_1,A_2\in\mathcal A$ such that $F_1\cap A_1=\emptyset$ and $F_2\cap A_2=\emptyset$. Since $\mathcal A$ is a $\ll$-filtered family, there exists $A_3\in\mathcal A$ such that $A_3\ll A_1, A_2$ (hence $A_3\subseteq A_1, A_2$). It then follows that $F_1\cap A_3\subseteq F_1\cap A_1=\emptyset$ and $F_2\cap A_3\subseteq F_2\cap A_2=\emptyset$, which implies that $F_1\cap A_3=F_2\cap A_3=\emptyset$. Thus  $F\cap A_3=(F_1\cup F_2)\cap A_3=(F_1\cap A_3)\cup (F_2\cap A_3)=\emptyset$, contradicting the assumption on $F$. Therefore,  $F$ is irreducible.
	\end{proof}

\section{Saturated well-filtered spaces}\label{S:three}

In this section, we show that well-filteredness  can be characterized by means of saturated sets, instead of compact saturated sets.

\begin{defi}
	A $T_0$ space $X$ is called \emph{saturated well-filtered}, if for any  $\{A_i:i\in I\}\subseteq_{flt} \mathcal S(X)$ and $U\in\mathcal O(X)$,
	$$\bigcap_{i\in I}A_i\subseteq U \ \Rightarrow \  \exists i_0\in I, A_{i_0}\subseteq U.$$
\end{defi}

\begin{defi}
	A $T_0$ space $X$ is called \emph{saturated $\omega$-well-filtered}, if for any  $\{A_n:n<\omega\}\subseteq_{flt} \mathcal S(X)$ and  $U\in\mathcal O(X)$,
	$$\bigcap_{n<\omega}A_n\subseteq U \ \Rightarrow \  \exists n_0<\omega, A_{n_0}\subseteq U.$$
\end{defi}

A countable family $\{A_n:n<\omega\}\subseteq\mathcal S(X)$ is called  a \emph{descending $\ll$-chain} if
$$A_0\gg A_1\gg A_2\gg\ldots\gg A_n\gg A_{n+1}\gg \ldots.$$

Analogous to Proposition \ref{chain}, we can prove the  following.
\begin{prop}
	A $T_0$ space $X$ is saturated $\omega$-well-filtered if and only if for any countable descending $\ll$-chain $\{A_n:n<\omega\}\subseteq\mathcal S(X)$
	and $U\in\mathcal O(X)$,
	$$\bigcap_{n<\omega}A_n\subseteq U\ \Rightarrow\ \exists n_0<\omega, \ A_{n_0}\subseteq U.$$
\end{prop}

\begin{prop}\label{comp}
	Let $X$ be a saturated well-filtered space. Then for any  $\{A_i:i\in I\}\subseteq_{flt} \mathcal S(X)\setminus\{\emptyset\}$, $\bigcap_{i\in I}A_i$ is a nonempty compact saturated set.
\end{prop}
\begin{proof}
	It is clear that  $\bigcap_{i\in I}A_i$ is saturated.
	Now suppose that $\bigcap_{i\in I}A_i=\emptyset$.
	Since $X$ is saturated well-filtered and $\emptyset$ is open, we have that $A_{i_0}\subseteq \emptyset$ for some $i_0\in I$, which contradicts that $A_{i_0}\neq\emptyset$. Thus
	$\bigcap_{i\in I}A_i\neq\emptyset$.
	
	Let $\{V_j:j\in J\}$ be an open cover of  $\bigcap_{i\in I}A_i$. As $X$ is saturated well-filtered,  there exists $i_0\in I$ such that $A_{i_0}\subseteq \bigcup_{j\in J}V_j$.  Since $\{A_i:i\in I\}\subseteq \mathcal S(X)$ is a $\ll$-filtered family, there exists $i_{1}\in I$ such that $A_{i_1}\ll A_{i_0}\subseteq \bigcup_{j\in J}V_j$. Then there exists $J_0\subseteq J^{(<\omega)}$ such that $A_{i_1}\subseteq \bigcup_{j\in J_0}V_{j}$.  It follows that $\bigcap_{i\in I}A_i\subseteq \bigcup_{j\in J_0}V_{j}$. Therefore,  $\bigcap_{i\in I}A_i$ is compact.
\end{proof}

Using a similar proof to that of Proposition \ref{comp}, we deduce the following.
\begin{prop}\label{chaincomp}
	Let $X$ be a saturated $\omega$-well-filtered space. Then for any  $\{A_n:n<\omega\}\subseteq_{flt} \mathcal S(X)\setminus\{\emptyset\}$, $\bigcap_{n<\omega}A_n$ is a nonempty compact saturated set.
\end{prop}

\begin{thm}\label{th1}
	The saturated $\omega$-well-filtered spaces are exactly the $\omega$-well-filtered spaces.
\end{thm}

\begin{proof}
Note that every descending chain in $\mathcal Q(X)$ is a descending $\ll$-chain in $\mathcal S(X)$. Hence every saturated $\omega$-well-filtered space is an $\omega$-well-filtered space.
	
	Now let $X$ be an $\omega$-well-filtered space. Suppose  that $\{A_n:n<\omega\}\subseteq\mathcal S(X)$ is a descending  $\ll$-chain, i.e.,
	$$A_0\gg A_1\gg A_2\gg\ldots\gg A_n\gg A_{n+1}\gg \ldots,$$
	and $U\in\mathcal O(X)$ such that  $\bigcap_{n<\omega} A_n\subseteq U$. We need to prove that $A_{n_0}\subseteq U$ for some $n_0<\omega$.
	Otherwise, $A_n\nsubseteq U$ for all $n<\omega$, that is, $A_n\cap (X\setminus U)\neq\emptyset$. Then by Lemma \ref{llrudin}, there exists a minimal (irreducible) closed set $F\subseteq X\setminus U$ such that $F\cap A_n\neq\emptyset$ for all $n<\omega$. Choose one $x_n\in F\cap A_n$ for each $n <\omega$, and let $H:=\{x_n:n<\omega\}$.
	
	\vspace{1mm}
	
	{Claim:} $H$ is compact.
	
	\vspace{1mm}
	
	Let $\{C_i:i\in I\}$ be a family of closed subsets of $X$ such that for any $J\in I^{(<\omega)}$, $H\cap\bigcap_{i\in J}C_i\neq\emptyset$. It needs to prove that $H\cap\bigcap_{i\in I}C_i\neq\emptyset$. We complete the proof
	by considering two cases.
	
	\vspace{1mm}
	
	(c1) $C_i\cap H$ is infinite for all  $i\in I$.
	
	\vspace{1mm}
	
	In this case, for each $ n<\omega$, there exists $k_n\geq n$ such that $x_{k_n}\in C_i$. Since $A_{k_n}\subseteq A_n$ and $x_{k_n}\in F\cap A_{k_n}$, we have that  $x_{k_n}\in C_i\cap F\cap A_{k_n}\subseteq C_i\cap F\cap  A_n\neq\emptyset$.  Thus $C_i\cap F$ is a closed set that intersects all $A_n$ ($n<\omega$). By the minimality of $F$, we have $F=C_i\cap F$, that is, $F\subseteq C_i$.  By the arbitrariness of $i\in I$,  it follows that $F\subseteq \bigcap_{i\in I} C_i$. Note that $H\subseteq F$, so  $H\cap\bigcap_{i\in I} C_i=H\neq\emptyset$.
	
	\vspace{1mm}
	
	(c2) $C_{i}\cap H$ is finite for some $i\in I$.
	
	\vspace{1mm}
	
	Let $i_0\in I$ such that $C_{i_0}\cap H$ is finite (hence compact).  Note that the family $\{C_i:i\in I\}$ satisfies that for any $J\in I^{(<\omega)}$, $H\cap C_{i_0}\cap\bigcap_{i\in J}C_i\neq\emptyset$.  Since $H\cap C_{i_0}$ is compact, we conclude
	$H\cap \bigcap_{i\in I}C_i=(H\cap C_{i_0})\cap \bigcap_{i\in I}C_i\neq\emptyset$.

	Now for each $n<\omega$,  let $H_n:=\{x_k:k\geq n\}$, which is compact by using a similar proof for $H$. Then $\{\ua H_k:k<\omega\}\subseteq_{flt}\mathcal Q(X)$ such that
	$\bigcap_{n<\omega} \ua H_n\subseteq \bigcap_{n<\omega}  A_n\subseteq U$. As $X$ is an $\omega$-well-filtered space,  there exists $n_0<\omega$ such that $\ua H_{n_0}\subseteq U$, which contradicts that $H_{n_0}\subseteq F\subseteq X\setminus U$.
\end{proof}

\begin{thm}\label{th2}
	The saturated well-filtered  spaces are exactly the well-filtered spaces.
\end{thm}
\begin{proof}
	Clearly, every saturated well-filtered space is a well-filtered space.
	
	Now assume that $X$ is a well-filtered space. Let $\mathcal A\subseteq_{flt} \mathcal S(X)$ and  $U\in\mathcal O(X)$ such that $\bigcap\mathcal A\subseteq U$.
	
	Define $$\widehat{\mathcal A}=\left\{\bigcap_{n<\omega}A_n: \forall n<\omega, A_n\in\mathcal A\text{ and } A_{n}\gg A_{{n+1}}\right\}.$$

	By Theorem \ref{th1} and the fact that every well-filtered space is $\omega$-well-filtered, we deduce that  $X$ is saturated $\omega$-well-filtered. Thus by Proposition \ref{chaincomp}, every member of $\widehat{\mathcal A}$ is  nonempty compact saturated.

	{Claim:} $\widehat{\mathcal A}$ is a filtered family.
	
	Let $\bigcap_{n<\omega}A_n, \bigcap_{n<\omega}B_n\in\widehat{\mathcal A}$, that is, $$A_0\gg A_1\gg A_2\gg\cdots \gg A_n\gg A_{n+1}\gg\ldots$$
	and $$B_0\gg B_1\gg B_2\gg\cdots \gg B_n\gg B_{n+1}\gg\cdots.$$
	
	(i) There exists $C_0\in \mathcal A$ such that $C_0\ll A_0, B_0$ because $\mathcal A$ is $\ll$-filtered.
	
	(ii)  If we have defined $\{C_{0}, C_{1},\cdots C_{n}\}$, then there  exists $A\in \mathcal A$ such that  $A\ll C_{n}, A_{{n+1}}, B_{{n+1}}$. Put $C_{{n+1}}:=A$.
	
	By Induction, we obtain a family $\{C_{n}:n<\omega\}\subseteq\mathcal A$ such that
	$$C_{0}\gg C_{1}\gg C_{2}\gg \cdots C_{n}\gg C_{{n+1}}\gg\cdots,$$
	and that $C_{n}\ll A_{n}, B_{n}$ (hence $C_n\subseteq A_n, B_n$) for all $n<\omega$.
	It follows that $\bigcap_{n<\omega}C_n\in\widehat{\mathcal A}$ and it is a lower bound of  $\left\{\bigcap_{n<\omega}A_n,\ \bigcap_{n<\omega}B_n\right\}$. Hence $\widehat{\mathcal A}$ is filtered.
	
	Since $X$ is well-filtered and  $\bigcap\widehat{\mathcal A}=\bigcap \mathcal A\subseteq U$, there is $\bigcap_{n<\omega}A_{n}\in\widehat{\mathcal A}$, where
	$$A_0\gg A_1\gg A_2\gg\cdots \gg A_n\gg A_{n+1}\gg\ldots,$$
	such that
	$\bigcap_{n<\omega}A_{n}\subseteq U$. By Theorem~\ref{th1},  $X$ is saturated $\omega$-well-filtered. Then there exists $n_0<\omega$ such that $A_{n_0}\subseteq U$ (note that $A_{n_0}\in \mathcal A$). Therefore, $X$ is well-filtered.
\end{proof}

\section{Open well-filtered spaces}\label{S:four}

In this section, we define another class of $T_0$ spaces, called open well-filtered spaces. It turns out that every well-filtered space is open well-filtered, and every core-compact open well-filtered space is sober. As an immediate consequence, we have that every core-compact well-filtered space is sober.

\begin{defi}
	A $T_0$ space is called \emph{open well-filtered}, if for any $\{U_i:i\in I\}\subseteq_{flt} \mathcal O(X)$ and $U\in\mathcal O(X)$,
	$$\bigcap_{i\in I}U_i\subseteq U \ \Rightarrow \  \exists i_0\in I,\ U_{i_0}\subseteq U.$$
\end{defi}

Note that every open set is saturated. By Theorem \ref{th2}, we obtain the following result.
\begin{rem}\label{wfiswfo}
	Every well-filtered space is open well-filtered.
\end{rem}

\begin{defi}
	A $T_0$ space $X$ is called \emph{open $\omega$-well-filtered}, if for any $\{U_n:n<\omega\}\subseteq_{flt} \mathcal O(X)$ and  $U\in\mathcal O(X)$,
	$$\bigcap_{n<\omega}U_n\subseteq U \ \Rightarrow \  \exists n_0<\omega,\ U_{n_0}\subseteq U.$$
\end{defi}

Using a similar proof to that of Proposition \ref{chain}, we deduce the following result.
\begin{prop}
	A $T_0$ space $X$ is open $\omega$-well-filtered if and only if for any countable descending $\ll$-chain $\{U_n:n<\omega\}\subseteq\mathcal O(X)$
	and $U\in\mathcal O(X)$,
	$$\bigcap_{n<\omega}U_n\subseteq U\ \Rightarrow\ \exists n_0<\omega, \ U_{n_0}\subseteq U.$$
\end{prop}

Analogous to Proposition \ref{comp}, we have the following two results.
\begin{prop} \label{p1}
	Let  $X$ be an open well-filtered space. Then for any $\{U_i:i\in I\}\subseteq_{flt} \mathcal O(X)$,  $\bigcap_{i\in I}U_i$ is a nonempty  compact saturated set.
\end{prop}

\begin{prop} \label{p2}
	Let  $X$ be an open $\omega$-well-filtered space. Then for any  $\{U_n:n<\omega\}\subseteq_{flt} \mathcal O(X)$,\ $\bigcap_{n<\omega}U_n$ is  a nonempty  compact saturated set.
\end{prop}

\begin{thm}\label{th}
	Every core-compact  open well-filtered space is sober.
\end{thm}
\begin{proof}
	Assume $X$ is a core-compact  open well-filtered space.  Let $A$ be an irreducible closed subset of $X$.
	Define
	$$\mathcal F_A:=\{U\in\mathcal O(X): U\cap A\neq\emptyset\}.$$
	
	\vspace{2mm}
	
	{Claim:} $\mathcal F_A$ is a $\ll$-filtered family.
	
	\vspace{2mm}
	
	Let $U_1,U_2\in\mathcal F_A$. Then $U_1\cap A\neq \emptyset\neq U_2\cap A$. As $A$ is irreducible, $U_1\cap U_2\cap A\neq\emptyset$, so there is an  $x\in A\cap U_1\cap U_2$. Since $X$ is core-compact, there exists $U_3\in\mathcal O(X)$ such that $x\in U_3\ll U_1\cap U_2$. Note that $x\in U_3\cap A\neq\emptyset$, so $U_3\in\mathcal F_A$. Hence, $\mathcal F_A$ is a $\ll$-filtered family.
	
	Since $X$ is open well-filtered, $A\cap\bigcap_{i\in I}\mathcal F_A\neq\emptyset$. Let $x_0\in A\cap\bigcap\mathcal F_A$. We show that $A=\cl(\{x_0\})$. Otherwise, $A\setminus\cl(\{x_0\})= A\cap (X\setminus\cl(\{x_0\}))\neq\emptyset$, implying that $X\setminus\cl(\{x_0\})\in\mathcal{F}_A$. It follows that $x_0\in\bigcap\mathcal{F_A}\subseteq X\setminus\cl(\{x_0\})$, a contradiction. Thus $A=\cl(\{x_0\})$. Since $X$ is a $T_0$ space, $\{x_0\}$ is unique. So  $X$ is sober.
\end{proof}

As a consequence of Remark \ref{wfiswfo} and Theorem \ref{th}, we obtain the following result.

\begin{cor}
	Every core-compact well-filtered space is sober.
\end{cor}

\begin{thm}\label{th3}
	Every core-compact open $\omega$-well-filtered space is locally compact.
\end{thm}
\begin{proof}
	Assume that $X$ is a core-compact open $\omega$-well-filtered space. Let $x\in X$ and $U\in\mathcal O(X)$ such that $x\in U$. Since $X$ is core-compact, there exists an open set $W\ll U$ such that $x\in W$ and a sequence of open sets $\{U_n:n<\omega\}$ such that
	$$U=U_0\gg U_1\gg U_2\gg U_3\gg\ldots\gg W.$$
	Let $Q=\bigcap_{n<\omega}U_n$. Since $X$ is open $\omega$-well-filtered and by Proposition \ref{p1}, $Q\in\mathcal Q(X)$ satisfies that 	$x\in W\subseteq Q\subseteq U$. Thus $X$ is locally compact.
\end{proof}

\begin{rem}
	In   \cite{Goubaultblog}, J. Goubault-Larrecq gives a  slighly different  proof for the above theorem.
	\end{rem}

As a corollary of Theorem \ref{th3}, we deduce the following result.

\begin{cor}\label{c1}
	A well-filtered space is core-compact if and only if it is locally compact.
\end{cor}

The following result is  a small variant of Kou's result that every well-filtered space is a $d$-space (see Proposition 2.4 in \cite{kou}).
\begin{propC}[\cite{kou}]\label{pro}
	Let $X$ be an $\omega$-well-filtered space. If  $D$ is a directed (under the specialization order) subset of $X$ with the cardinality $|D|\le \omega$,  then $\bigvee D$ exists.
\end{propC}

Let $P$ be a poset.	A subset $U$ of  $P$ is \emph{Scott open} if
(i) $U=\mathord{\uparrow}U$ and (ii) for any directed subset $D$ of $P$ for
which $\bigvee D$ exists, $\bigvee D\in U$ implies $D\cap
U\neq\emptyset$. All Scott open subsets of $P$ form a topology,
called the \emph{Scott topology} on $P$ and
denoted by $\sigma(P)$. The space $\Sigma P=(P,\sigma(P))$ is called the
\emph{Scott space} of $P$.

\begin{exa}\label{exam}
	Let $\mathbb J=\mathbb N\times(\mathbb N\cup\{\omega\})$ be the Johnstone's dcpo and $\mathbb N=\{1,2,3\ldots\}$ with the usual order. Let $P=\mathbb J\cup \mathbb N$. For any $x,y\in P$, define $x\leq y$ (refer to Figure 1) if one of the following conditions holds:
	\begin{enumerate}[(i)]
		\item $x,y\in \mathbb N$ and $x\leq y$ in $\mathbb N$;
		\item $x,y\in \mathbb J$ and $x\leq y$ in $\mathbb J$;
		\item $x\in \mathbb N,y\in\mathbb J$ and $y= (x,\omega)$.
	\end{enumerate}
	\begin{figure}[ht]
		\centering
		\includegraphics[scale=0.14]{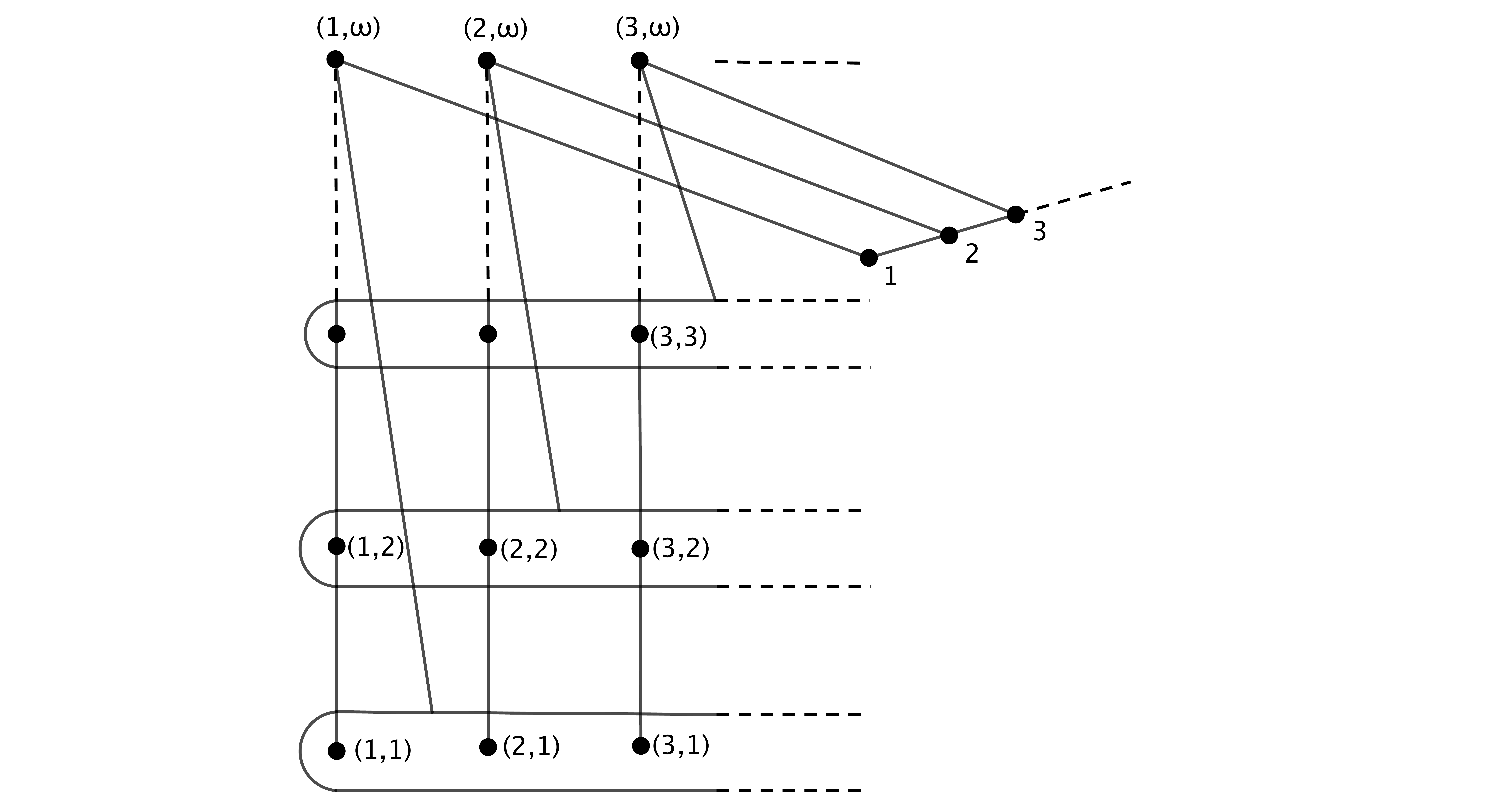}
		\caption{The poset $P$}
	\end{figure}
	
Next, we show that $\Sigma\,P$	is an open well-filtered space. The following conclusions on $P$ will be used later.
	
	\begin{enumerate}[(c1)]
		\item For any directed subset $D$ of $P$, if $\bigvee D$ exists, then  $\bigvee D\in D$ or $D\subseteq\mathbb J$.
		
		\vspace{1mm}

		In fact, assume $D\nsubseteq \mathbb J$, that is, $D\cap\mathbb N\neq\emptyset$. We prove $\bigvee D\in D$ by considering  the following two cases:
			
			Case 1: $D\subseteq\mathbb N$. In this case, it is trivial that $\bigvee D\in D$.
			
			Case 2:  $D\nsubseteq \mathbb N$, that is, $D\cap \mathbb{N}\neq\emptyset$ and $D\cap \mathbb J\neq\emptyset$. Let $k\in D\cap\mathbb N$ and $(m,n)\in D\cap\mathbb J$. Since  $D$ is directed, there exists $d\in D$ such that $k, (m,n)\leq d$. It forces that  $d=(n_0,\omega)$ for some $n_0\in\mathbb N$, which is a maximal point of $P$, so $d=\bigvee D\in D$ (note that the maximal point of a directed set  is exactly the least upper bound of the set).
	\end{enumerate}
	
	\vspace{1mm}
	
	\begin{enumerate}[(c2)]
		\item [(c2)]  For any $U\in\sigma(P)\setminus\{\emptyset\}$, there exists a minimal $n_U\in\mathbb N$ such that $(n,\omega)\in U$ for all $n\geq n_U$,  that is, $n_U=\min\{k\in\mathbb N:\forall n\geq k, (n,w)\in U\}$ exists in $\mathbb N$ (refer to Figure 1).
	
	\vspace{1mm}

		Assume, on the contrary, that $n_U$ does not exist in $\mathbb N$. Then there exist infinitely countable numbers $n_1,n_2,n_3,\ldots$ in $\mathbb N$ such that $(n_k, \omega)\notin U$ for all $k\in\mathbb N$. Since $U$ is a nonempty upper set, there exists $m\in\mathbb N$ such that $(m,\omega)\in U$. Note that $\{(m,n_k): k\in\mathbb N\}$ is a directed set such that $\bigvee \{(m,n_k): k\in\mathbb N\}=(m,\omega)\in U$. Then there exists $k_0\in\mathbb N$ such that $(m,n_{k_0})\in U$. Since $(n_{k_0},\omega)\geq (m,n_{k_0})\in U$ and $U$ is an upper set,  it holds that  $(n_{k_0},\omega)\in U$, a contradiction.
\end{enumerate}

	\vspace{1mm}
	
	\begin{enumerate}[(c3)]
		\item For any $U\in \sigma(P)\setminus\{\emptyset\}$ and  $\forall (n,\omega)\in U$, there exists a minimal $\phi_U(n)\in\mathbb N$ such that $(n,\phi_U(n))\in U$, i.e., $\phi_U(n)=\min\{m\in \mathbb N: (n,m)\in U\}$ exists.
		
		\vspace{1mm}
			Suppose  $(n,\omega)\in U$. Then $\{(n,m): m\in\mathbb N\}$ is a directed subset of $P$ such that  $\bigvee \{(n,m): m\in\mathbb N\}=(n,\omega)\in U$. Since $U$ is Scott open,  there exists  $m_0\in \mathbb N$ such that $(n,m_0)\in U$.
			Thus  $m_0\in \{m\in \mathbb N: (n,m)\in U\}\neq\emptyset$, which means that  $\varphi_U(n)=\min\{m\in \mathbb N: (n,m)\in U\}$ exists.
	\end{enumerate}

	\vspace{1mm}

	\begin{enumerate}[(c4)]
		\item  For any $U,V\in \sigma(P)$, $U\ll V$ if and only if  $U=\emptyset$.
		
		\vspace{1mm}
		If $U=\emptyset$, then trivially $U=\emptyset\ll V$.
			Now assume $U\ll V$ and $U\neq\emptyset$.
			By using (c2) and (c3), for each $n\geq n_U$, define $U_n=P\setminus \bigcup_{k\geq n}\da(k,\phi_U(k))$.
			It is trivial that  $\bigcup_{k\geq n}\da(k,\phi_U(k))$ is a Scott closed subset of $P$, which means that $U_n\in\sigma(P)$.
		Since $\bigcup_{n\geq n_U}P\setminus \bigcup_{k\geq n}\da(k,\phi_U(k))=P\setminus \bigcap_{n\geq n_U}\bigcup_{k\geq n}\da(k,\phi_U(k))=P\setminus\emptyset=P$. Thus the family $\{U_n:n\geq n_U\}$ is a directed open cover of $P$, hence a directed open cover of $V$.
			By assumption that $U\ll V$, there exists $n_0\geq n_U$ such that $U\subseteq U_{n_0}$. By the definition of $n_U$, it follows that  $(n_0,\omega)\in U$. Thus by the definition of $\phi_U$ in  (c3), we have that $(n_0,\phi_U(n_0))\in U$, but $(n_0,\phi_U(n_0))\notin P\setminus\da(n_0,\phi_U(n_0))\supseteq U_{n_0}$, which implies $(n_0,\phi_U(n_0))\notin U_{n_0}$.  It follows that $(n_0,\phi_U(n_0))\in U\setminus U_{n_0}$, contradicting that $U\subseteq U_{n_0}$.
	\end{enumerate}

\vspace{1mm}

	By (c4),  we conclude that every $\ll$-filtered family $\mathcal F$ of Scott open subsets of $P$ contains $\emptyset$, the minimal element in $\mathcal F$.
 This implies that $\Sigma P$ is an open well-filtered space (hence an open $\omega$-well-filtered space).  Since $\bigvee\mathbb N$ does not exist in $P$ and by Proposition \ref{pro},  $\Sigma P$ is not an $\omega$-well-filtered space.
\end{exa}

Recall that a $T_0$ space $X$ is called a \emph{$d$-space} if $X$ is a dcpo in its specialization order and each open subset of $X$ is Scott open in the specialization order. It is well-known that every well-filtered space is a $d$-space \cite{Jean-2013}.
From Example \ref{exam}, we  have the following conclusions.
\begin{enumerate}
	\item An open well-filterd  space need not  be a $d$-space.
	\item An open well-filtered space need not be an $\omega$-well-filtered space.
\end{enumerate}

	\vspace{2mm}
	
A summary on the relations among kinds of  well-filtered spaces is shown in Figure 2.

\begin{figure}[ht]
$$\xymatrix @R=2.6pc @C=2pc{
	\txt{\small sober}\ar@{->}[d]&& \\
	\txt{\small well-filtered}\ar@{<->}[r]\ar@{->}[d]
	&\txt{\small saturated\\ \small  well-filtered}\ar@{->}[r]\ar@{->}[d]
	&\txt{\small open\\ \small  well-filtered}\ar@{->}[d]\ar@{->}[ull]_{\txt{ \small\ \ \ \   core-compact}}\\
	\txt{\small saturated \\ \small $\omega$-well-filtered}\ar@{<->}[r]
	&\txt{\small$\omega$-well-filtered}\ar@{->}[r]
	&\txt{\small open\\ \small  $\omega$-well-filtered}\ar@{->}[rr]^{\txt{\small core-compact}}&
&\txt{ \small locally\\ \small compact}}
$$
\caption{The relations among various types of well-filtered spaces}
\end{figure}
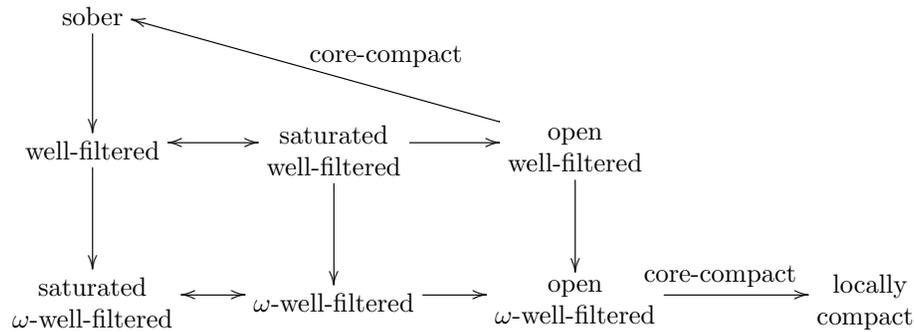

	\section{Acknowledgment}
	\noindent We would like to thank the reviewers for giving us valuable comments and suggestions for improving the manuscript.


\end{document}